\documentclass[11pt,english]{amsart}
\usepackage[T1]{fontenc}
\usepackage[latin1]{inputenc}
\usepackage{latexsym}
\usepackage{amsthm}
\usepackage{amssymb}

\numberwithin{equation}{section} %% Comment out for sequentially-numbered
\numberwithin{figure}{section} %% Comment out for sequentially-numbered
\theoremstyle{plain}
\theoremstyle{plain}
\newtheorem{thm}{Theorem}
  \theoremstyle{plain}
  \newtheorem{prop}[thm]{Proposition}
  \theoremstyle{remark}
  \newtheorem{acknowledgement}[thm]{Acknowledgement}
  \theoremstyle{definition}
  \newtheorem{defn}[thm]{Definition}
  \theoremstyle{plain}
  \newtheorem{lem}[thm]{Lemma}
  \theoremstyle{remark}
  \newtheorem{notation}[thm]{Notation}
  \theoremstyle{remark}
  \newtheorem{rem}[thm]{Remark}
  \theoremstyle{plain}
  \newtheorem{cor}[thm]{Corollary}

\AtBeginDocument{
  
}

\usepackage{babel}
\addto\extrasfrench{\providecommand{\fg}{\ifdim\lastskip>\z@\unskip\fi~\frqq}}

\begin{document}
\newcommand{\Alb}{{\rm Alb}}

\newcommand{\Jac}{{\rm Jac}}

\newcommand{\Hom}{{\rm Hom}}

\newcommand{\End}{{\rm End}}

\newcommand{\Aut}{{\rm Aut}}

\newcommand{\NS}{{\rm NS}}

\newcommand{\SSm}{{\rm S}}

\newcommand{\Ker}{{\rm Ker}}

\title{The Fano surface of the Klein cubic threefold}

\author{Xavier Roulleau}

\maketitle

Key-words: Fano surface of a cubic threefold, Automorphisms, Maximal
Picard number, Intermediate Jacobians.

MSC (2000): 14J29, 14J50 (primary); 14J70, 32G20 (secondary).
\begin{abstract}
We prove that the Klein cubic threefold $F$ is the only smooth cubic
threefold which has an automorphism of order $11$. We compute the
period lattice of the intermediate Jacobian of $F$ and study its
Fano surface $S$. We compute also the set of fibrations of $S$ onto
a curve of positive genus and the intersection between the fibres
of these fibrations. These fibres generate an index $2$ sub-group
of the Néron-Severi group and we obtain a set of generators of this
group. The Néron-Severi group of $S$ has rank $25=h^{1,1}$ and discriminant
$11^{10}$.
\end{abstract}

\section{Introduction.}

Let $F\hookrightarrow\mathbb{P}^{4}$ be a smooth cubic threefold.
Its intermediate Jacobian \[
J(F):=H^{2,1}(F,\mathbb{C})^{*}/H_{3}(F,\mathbb{Z})\]
 is a $5$ dimensional principally polarized abelian variety $(J(F),\Theta)$
that plays in the analysis of curves on $F$ a similar role to the
one played by the Jacobian variety in the study of divisors on a curve.
\\
The set of lines on $F$ is parametrized by the so-called Fano
surface of $F$ which is a smooth surface of general type that we
will denote by $S$. The Abel-Jacobi map $\vartheta:S\rightarrow J(F)$
is an embedding that induces an isomorphism $\Alb(S)\rightarrow J(F)$
where \[
\Alb(S):=H^{0}(\Omega_{S})^{*}/H_{1}(S,\mathbb{Z})\]
 is the Albanese variety of $S$, $\Omega_{S}$ is the cotangent bundle
and $H^{0}(\Omega_{S}):=H^{0}(S,\Omega_{S})$ (see \cite{Clemens}
0.6, 0.8 for details). 

The tangent bundle theorem (\cite{Clemens} Theorem 12.37) enables
to recover the cubic $F$ from its Fano surface. Moreover it gives
a natural isomorphism between the spaces $H^{0}(\Omega_{S})$ and
$H^{0}(F,\mathcal{O}_{F}(1))=H^{0}(\mathbb{P}^{4},\mathcal{O}_{\mathbb{P}^{4}}(1))$.
As we mainly work with the Fano surface, we will identify the basis
$x_{1},\dots,x_{5}$ of $H^{0}(\mathbb{P}^{4},\mathcal{O}_{\mathbb{P}^{4}}(1))$
with elements of $H^{0}(\Omega_{S})$ (for more explanations about
these facts, see the discussion after Proposition \ref{groupe d'ordre 11}).
We will also identify the abelian variety $J(F)$ with $\Alb(S)$.

In \cite{Roulleau}, we give the classification of elliptic curve
configurations on a Fano surface. It is proved that this classification
is equivalent to the classification of the automorphism sub-groups
of $S$ that are generated by certain involutions. Moreover, it is
also proved that the automorphism groups of a cubic and its Fano surface
are isomorphic. \\
In the present paper, we pursue the study of these groups. By Lemma
\ref{odre du groupe d'autom} below, the order of the automorphism
group $\Aut(S)$ of the Fano surface divides $11\cdot7\cdot5^{2}3^{6}2^{23}.$
\\
This legitimates the study of the Fano surfaces which have automorphisms
of order $7$ or $11$. A. Adler \cite{Adler} has proved that the
automorphism group of the Klein cubic: \[
F:\, x_{1}x_{5}^{2}+x_{5}x_{3}^{2}+x_{3}x_{4}^{2}+x_{4}x_{2}^{2}+x_{2}x_{1}^{2}=0\]
is isomorphic to $PSL_{2}(\mathbb{F}_{11})$. We prove that:
\begin{prop}
\label{groupe d'ordre 11}A smooth cubic threefold has no automorphism
of order $7$.\\
The Klein cubic is the only one smooth cubic threefold which has
an automorphism of order $11$.
\end{prop}
If a curve admits a sufficiently large group of automorphisms, Bolza
has given a method to compute a period matrix of its Jacobian (see
\cite{Birkenhake}, 11.7). As for the case of curves, we use the fact
that the Klein cubic $F$ admits a large group of automorphisms to
compute the period lattice of its intermediate Jacobian $J(F)$ or,
what is the same thing, the period lattice $H_{1}(S,\mathbb{Z})\subset H^{0}(\Omega_{S})^{*}$
of the two dimensional variety $S$.

To state the main result of this work, we need some notations:\\
Let be $\nu=\frac{-1+i\sqrt{11}}{2}$ where $i\in\mathbb{C},\, i^{2}=-1$
and let $\mathbb{E}$ be the elliptic curve $\mathbb{E}=\mathbb{C}/\mathbb{Z}[\nu]$.
Let us denote by $e_{1},\dots,e_{5}\in H^{0}(\Omega_{S})^{*}$ the
dual basis of $x_{1},\dots,x_{5}$. Let be $\xi=e^{\frac{2i\pi}{11}}$
and for $k\in\mathbb{Z}/11\mathbb{Z}$, let $v_{k}\in H^{0}(\Omega_{S})^{*}$
be : \[
v_{k}=\xi^{k}e_{1}+\xi^{9k}e_{2}+\xi^{3k}e_{3}+\xi^{4k}e_{4}+\xi^{5k}e_{5}.\]

\begin{thm}
\label{le r=0000E9seau de A de klein}1) The period lattice $H_{1}(S,\mathbb{Z})\subset H^{0}(\Omega_{S})^{*}$
of the Fano surface of the Klein cubic is equal to: \[
\Lambda=\frac{\mathbb{Z}[\nu]}{1+2\nu}(v_{0}-3v_{1}+3v_{2}-v_{3})+\frac{\mathbb{Z}[\nu]}{1+2\nu}(v_{1}-3v_{2}+3v_{3}-v_{4})+\bigoplus_{k=0}^{2}\mathbb{Z}[\nu]v_{k}\]
and the first Chern class of the Theta divisor is: $c_{1}(\Theta)=\frac{i}{\sqrt{11}}\sum_{j=1}^{j=5}dx_{j}\wedge d\bar{x}_{j}$.\\
2) The Néron-Severi group $\NS(S)$ of $S$ has rank $25=h^{1,1}(S)$
and discriminant $11^{10}$. \\
3) Let $\vartheta:S\rightarrow\Alb(S)$ be an Albanese morphism
and $\NS(\Alb(S))$ be the Néron-Severi group of $\Alb(S)$. The set
of numerical classes of fibres of connected fibrations of $S$ onto
a curve of positive genus is in natural bijection with $\mathbb{P}^{4}(\mathbb{Q}(\nu))$
and the fibres of these fibrations generate rank $25$ sub-lattice
$\vartheta^{*}\NS(\Alb(S))$.
\end{thm}
Actually, the period lattice is equal to $c\Lambda$ where $c\in\mathbb{C}$
is a constant, but we can normalize $e_{1},\dots,e_{5}$ in such a
way that $c=1$, see Remark \ref{rem:Up-to-a}.

A set of generators of $\NS(S)$ and a formula for their intersection
numbers are given in Theorem \ref{la forme Q pour 11}. 

We remark that $J(F)\simeq\Alb(S)$ is isomorphic to $\mathbb{E}^{5}$
but by \cite{Clemens} (0.12), this isomorphism is not an isomorphism
of principally polarized abelian varieties (p.p.a.v.). The fact that
$J(F)$ is isomorphic to $\mathbb{E}^{5}$ is proved in \cite{Adler1}
in a different way.

The main properties used to prove Theorem \ref{le r=0000E9seau de A de klein}
are the fact that the action of the group $\Aut(S)$ on $\Alb(S)$
preserves the polarization $\Theta$ and the fact that the class of
$S\hookrightarrow\Alb(S)$ is equal to $\frac{1}{3!}\Theta^{3}$.
We use also the knowledge of the analytic representation of the automorphisms
of the p.p.a.v. $(\Alb(S),\Theta)$.

To close this introduction, let us mention two known facts on this
cubic: (1) the cotangent sheaf of its Fano surface is ample \cite{Roulleau},
(2) as the plane Klein quartic, the Klein cubic threefold has a modular
interpretation \cite{Gross} (about the analogy with the Klein quartic,
see also Remark \ref{quartique de Klein}). 
\begin{acknowledgement}
The author wishes to thank the referee for his numerous valuable remarks
on a previous version of this paper.\\
This work was done at the Max-Planck Institut of Bonn, which is
also acknowledged.
\end{acknowledgement}

\section{the unique cubic with an automorphism of order $11$.\label{le parag auto 11}\label{parag rappels notations auto}}

Let us recall some facts proved in \cite{Roulleau} and fix the notations
and conventions:
\begin{defn}
A morphism between two abelian varieties $f:A\rightarrow B$ is the
composition of a homomorphism of Abelian varieties $g:A\rightarrow B$
and a translation. We call $g$ the \emph{homomorphism part} of $f$.
The differential $dg:T_{A}(0)\rightarrow T_{B}(0)$ at the point $0$
is called the \emph{analytic representation} of both $f$ and $g$,
where $T_{A}(0)$ and $T_{B}(0)$ denote the tangent spaces of $A$
and $B$ at $0$.
\end{defn}
An automorphism $f$ of a smooth cubic $F\hookrightarrow\mathbb{P}^{4}$
preserves the lines on $F$ and induces an automorphism $\rho(f)$
of the Fano surface $S$. \\
An automorphism $\sigma$ of $S$ induces an automorphism $\sigma'$
of the Albanese variety of $S$ such that the following diagram:\[
\begin{array}{ccc}
S & \stackrel{\vartheta}{\rightarrow} & \Alb(S)\\
\sigma\downarrow &  & \sigma'\downarrow\\
S & \stackrel{\vartheta}{\rightarrow} & \Alb(S)\end{array}\]
is commutative (where $\vartheta:S\rightarrow\Alb(S)$ is a fixed
Albanese morphism). The tangent space of the Albanese variety $\Alb(S)$
is $H^{0}(\Omega_{S})^{*}$. We denote by $M_{\sigma}\in GL(H^{0}(\Omega_{S})^{*})$
the analytic representation of $\sigma'$ : it is the dual of the
pull-back $\sigma^{*}:H^{0}(\Omega_{S})\rightarrow H^{0}(\Omega_{S})$.
We denote by \[
q:GL(H^{0}(\Omega_{S})^{*})\rightarrow PGL(H^{0}(\Omega_{S})^{*})\]
the natural quotient map. 
\begin{thm}
\label{isomorphisme entre aut S et aut F}1) For $\sigma\in\Aut(S)$,
the homomorphism part of the automorphism $\sigma'$ is an automorphism
of the p.p.a.v. $(\Alb(S),\Theta)$.\\
2) Let $M$ be the analytic representation of an automorphism of
$(\Alb(S),\Theta)$, then $q(M)$ is an automorphism of $F\hookrightarrow\mathbb{P}^{4}=\mathbb{P}(H^{0}(\Omega_{S})^{*})$.\\
3) The morphism $\rho:\Aut(F)\rightarrow\Aut(S)$ is an isomorphism
and its inverse is the morphism : $\Aut(F)\rightarrow\Aut(S);\,\sigma\mapsto q(M_{\sigma})$.
\\
4) The group $\Aut(S)$ is a sub-group of $\Aut(\Alb(S),\Theta)$.
If $S$ is a generic Fano surface, then: \[
\Aut(F)\simeq\Aut(S)\simeq\Aut(\Alb(S),\Theta)/\left\langle [-1]\right\rangle .\]

\end{thm}
When $C$ is an non-hyperelliptic curve and $(J(C),\Theta)$ is its
jacobian, there is an isomorphism $\Aut(C)\simeq\Aut(J(C),\Theta)/\left\langle [-1]\right\rangle $
(\cite{Birkenhake} Chap.11, exercise 19). Result 4) of Theorem \ref{isomorphisme entre aut S et aut F}
is thus the analogue for a cubic and its intermediate Jacobian.
\begin{proof}
Part 1) and 3) are proved in \cite{Roulleau}, they imply that the
morphism:\[
\begin{array}{ccc}
\Aut(S) & \rightarrow & \Aut(\Alb(S),\Theta)\\
\sigma & \mapsto & \sigma''\end{array}\]
is injective, where $\sigma''$ is the homomorphism part of $\sigma'$.\\
 Let us denote by $B_{x}X$ the blow-up at the point $x$ of a
variety $X$. By \cite{Beauville}, the point $0$ of $\Alb(S)$ is
the unique singularity of the divisor $\Theta$ : it is thus preserved
by any automorphism $\tau$ of the p.p.a.v. $(\Alb(S),\Theta)$. \\
The automorphism $\tau$ induces an automorphism of $B_{0}\Theta$
and $B_{0}\Alb(S)$. Let $M$ denotes the analytic representation
of $\tau$. Let $E$ be the exceptional divisor of $B_{0}\Theta$.
The exceptional divisor of $B_{0}\Alb(S)$ is $\mathbb{P}(H^{0}(\Omega_{S})^{*})$
and we consider the following diagram:\[
\begin{array}{ccc}
E & \rightarrow & B_{0}\Theta\\
\downarrow &  & \downarrow\\
\mathbb{P}(H^{0}(\Omega_{S})^{*}) & \rightarrow & B_{0}\Alb(S)\end{array}\]
 where all the maps are embeddings. The action on $E$ of $\tau$
is obtained by restriction of the action of $\tau$ on $\mathbb{P}(H^{0}(\Omega_{S})^{*})$
(that is the space of tangent directions to the point $0$ of $\Alb(S)$)
; this last action is given by the automorphism $q(M)$ that is the
projectivization of the differential of the automorphism $\tau$ at
$0$. \\
Now, by \cite{Beauville} théorème p. 190, the exceptional divisor
$E\hookrightarrow\mathbb{P}(H^{0}(\Omega_{S})^{*})$ is the cubic
$F\hookrightarrow\mathbb{P}(H^{0}(\Omega_{S})^{*})$ itself, thus
property 2) holds. \\
Suppose that $q(M)$ is the identity. There exists a root of unity
$\lambda$ such that $M$ is the multiplication by $\lambda$. By
\cite{Birkenhake}, Corollary 13.3.5, the order of $\lambda$ is $1,2,3,4$
or $6$, and if the order $d$ is $3,4$ or $6$, then $\Alb(S)$
is isomorphic to $E^{5}$ where $E$ is the unique elliptic curve
with an automorphism of order $d$. \\
On the other hand, the divisor $\Theta$ is symmetric : $[-1]^{*}\Theta=\Theta$
(see \cite{Clemens}), hence if the cubic threefold is generic, $q(M)$
is the identity if and only if $\tau$ is equal to $[1]$ or $[-1]$
and: $\Aut(F)\simeq\Aut(S)\simeq\Aut(\Alb(S),\Theta)/\left\langle [-1]\right\rangle $.\end{proof}
\begin{lem}
\label{odre du groupe d'autom}The order of the group $\Aut(S)$ divides
$11\cdot7\cdot5^{2}3^{6}2^{23}$.\end{lem}
\begin{proof}
The automorphism group of a p.p.a.v. acts faithfully on the group
of $n$-torsion points for $n\geq3$ (\cite{Birkenhake}, Corollary
5.1.10). Thus the order of the group of automorphisms of a $5$ dimensional
p.p.a.v. must divides $a_{n}=\#GL_{10}(\mathbb{Z}/n\mathbb{Z})$ for
$n\geq3$. In particular, it divides \[
11\cdot7\cdot5^{2}3^{6}2^{23}=gcd(a_{3},a_{5},a_{7},a_{11}).\]
Theorem \ref{isomorphisme entre aut S et aut F} implies then that
the order of $\Aut(S)$ divides $11\cdot7\cdot5^{2}3^{6}2^{23}$. 
\end{proof}
Let us prove the following:
\begin{prop}
\label{groupe d'ordre 11}The Klein cubic is the only one smooth cubic
threefold which has an automorphism of order $11$. There is no smooth
cubic threefold with possesses an automorphism of order $7$.
\end{prop}
Let $F\hookrightarrow\mathbb{P}^{4}$ be a smooth cubic threefold
and $S$ be its Fano surface. In the Introduction, we mention that
the cubic threefold $F$ can be recovered knowing only the surface
$S$. This important, but non-trivial, result is called the Tangent
Bundle Theorem and is due to Fano, Clemens-Griffiths and Tyurin (Beauville
also gives another proof in \cite{Beauville}). We give more explanation
about that result ; it will explain how we identify the spaces $H^{0}(\mathbb{P}^{4},\mathcal{O}_{\mathbb{P}^{4}})$
and $H^{0}(\Omega_{S})$ and the main ideas of the proof of Proposition
\ref{groupe d'ordre 11}:\\
Let us consider the natural morphism of vector spaces:\[
Ev:\oplus_{n\in\mathbb{N}}\SSm^{n}H^{0}(\Omega_{S})\rightarrow\oplus_{n\in\mathbb{N}}H^{0}(S,\SSm^{n}\Omega_{S})\]
given by the natural maps on each pieces of the graduation. The Tangent
Bundle Theorem can be formulated as follows:\\
The kernel of $Ev$ is an ideal of the ring $\oplus_{n\in\mathbb{N}}\SSm^{n}H^{0}(\Omega_{S})$
generated by a cubic $F_{eq}\in\SSm^{3}H^{0}(\Omega_{S})$ and the
cubic threefold $\{F_{eq}=0\}\hookrightarrow\mathbb{P}(H^{0}(\Omega_{S})^{*})$
is (isomorphic to) the original cubic $F\hookrightarrow\mathbb{P}^{4}$.\\
By this Theorem, the homogenous coordinates $x_{1},\dots,x_{5}\in H^{0}(\mathbb{P}^{4},\mathcal{O}_{\mathbb{P}^{4}}(1))$
of $\mathbb{P}^{4}$ are also a basis of the global holomorphic $1$-forms
of $S$.
\begin{proof}
(of Proposition \ref{groupe d'ordre 11}). The intermediate Jacobian
of a smooth cubic threefold with an order $11$ automorphism is a
$5$ dimensional p.p.a.v. that possesses an automorphism of order
11. By the theory of complex multiplication there are only four such
principally abelian varieties, they are denoted by $X_{1},\dots,X_{4}$
in \cite{Gonzalez}. 

By \cite{Clemens}, an intermediate Jacobian is not a Jacobian of
a curve, but by Theorem $2$ of \cite{Gonzalez}, the p.p.a.v. $X_{1}$
and $X_{2}$ are Jacobians of curves, thus we can eliminate $X_{1}$
and $X_{2}$.

The abelian variety $X_{3}$ has an automorphism $\tau'$ of order
$11$ such that the eigenvalues of its analytic representation $M$
are $\{\xi,\xi^{2},\xi^{3},\xi^{5},\xi^{7}\}$, $\xi=e^{\frac{2i\pi}{11}}$.\\
Suppose that $X_{3}$ is the intermediate Jacobian of a cubic threefold
$F\hookrightarrow\mathbb{P}^{4}$. By Theorem \ref{isomorphisme entre aut S et aut F},
the morphism $^{t}M$ acts on $H^{0}(\mathbb{P}^{4},\mathcal{O}_{\mathbb{P}^{4}}(1))$
and the projectivized morphism $q(M)$ is an automorphism of $F\hookrightarrow\mathbb{P}^{4}$.
\\
Let $\SSm^{3}(^{t}M)$ be the action of $^{t}M$ on \[
H^{0}(\mathbb{P}^{4},\mathcal{O}_{\mathbb{P}^{4}}(3))=\SSm^{3}H^{0}(\mathbb{P}^{4},\mathcal{O}_{\mathbb{P}^{4}}(1)).\]
An equation of $F$ is an eigenvector for $\SSm^{3}(^{t}M)$. We can
easily compute the eigenspaces of $\SSm^{3}(^{t}M)$ and check that
no one contains the equation of a smooth cubic threefold, thus $X_{3}$
cannot be an intermediate Jacobian.

The p.p.a.v. $X_{1},X_{2}$ and $X_{3}$ are not intermediate Jacobians
and by the Torelli Theorem \cite{Clemens}, the association $F\rightarrow(J(F),\Theta)$
is injective, hence the p.p.a.v. $X_{4}$ is the intermediate Jacobian
of the Klein cubic and this cubic is (up to isomorphism) the only
one smooth cubic which has an order $11$ automorphism.

Let us prove that there are no smooth cubic threefolds with an automorphism
of order $7$.\\
By the Theorems 13.1.2. and 13.2.8. and Proposition 13.2.5. of
\cite{Birkenhake}, an automorphism of order $7$ of a five dimensional
Abelian variety has eigenvalues $1,1,\mu,\mu^{a},\mu^{b}$ where $\mu$
is a primitive $7$-th root of unity and $a,b$ are integers such
that $\{\mu,\mu^{a},\mu^{b}\}$ is a set containing three pairwise
non-complex conjugate primitive roots of unity. \\
Up to the change of $\mu$ by a power, there are two possibilities:\\
The first case is $a=2$ and $b=3$. For $c\in\mathbb{Z}/7\mathbb{Z}$,
let us denote by $\chi_{c}$ the representation $\mathbb{C}\rightarrow\mathbb{C};\, t\mapsto\mu^{c}t$
of $\mathbb{Z}/7\mathbb{Z}$. The third symmetric product of the representation:\[
\begin{array}{ccc}
H^{0}(\mathbb{P}^{4},\mathcal{O}_{\mathbb{P}^{4}}(1)) & \rightarrow & H^{0}(\mathbb{P}^{4},\mathcal{O}_{\mathbb{P}^{4}}(1))\\
x_{1},x_{2},x_{3},x_{4},x_{5} & \mapsto & \mu x_{1},\mu^{2}x_{2},\mu^{3}x_{3},x_{4},x_{5}\end{array}\]
decomposes as the direct sum:\[
6\chi_{0}+4\chi_{1}+6\chi_{2}+6\chi_{3}+5\chi_{4}+4\chi_{5}+4\chi_{6}.\]
By example, the space corresponding to $6\chi_{0}$ is generated by
the forms $x_{4}^{3},x_{4}^{2}x_{5},x_{4}x_{5}^{2},x_{5}^{3},x_{2}^{2}x_{3},x_{3}^{2}x_{1}$.
No element of this space is an equation of a smooth cubic threefold.
In the similar way, we can check that the other factors do not give
a smooth cubic.\\
The second case is $a=2$ and $b=4$. In the same manner, we can
check that we do not obtain a smooth cubic in that case.
\end{proof}

\section{The period lattice of the intermediate Jacobian of the Klein cubic.}

Let $F$ be the Klein cubic:\[
x_{1}x_{5}^{2}+x_{5}x_{3}^{2}+x_{3}x_{4}^{2}+x_{4}x_{2}^{2}+x_{2}x_{1}^{2}=0\]
and let $S$ be its Fano surface. Let $\vartheta:S\rightarrow\Alb(S)$
be a fixed Albanese morphism; it is an embedding. Let us compute the
period lattice $H_{1}(S,\mathbb{Z})\subset H^{0}(\Omega_{S})^{*}$
of the Albanese variety of $S$.

The order $5$ automorphism:\[
g:\,(z_{1}:z_{2}:z_{3}:z_{4}:z_{5})\mapsto(z_{5}:z_{1}:z_{4}:z_{2}:z_{3})\]
acts on $F$. Let $\sigma=\rho(g)$ be the automorphism of $S$ defined
in paragraph \ref{parag rappels notations auto}. By Theorem \ref{isomorphisme entre aut S et aut F},
there exists a $5$-th root of unity $\theta$ such that $M_{\sigma}\in GL(H^{0}(\Omega_{S})^{*})$
is equal to: \[
M_{\sigma}:(z_{1},z_{2},z_{3},z_{4},z_{5})\mapsto\theta(z_{5},z_{1},z_{4},z_{2},z_{3})\]
in the basis $e_{1},\dots,e_{5}$ dual to $x_{1},\dots,x_{5}$. Since
the Klein cubic $F$ and $g$ are defined over $\mathbb{Q}$ and the
Fano surface contains points defined over $\mathbb{Q}$, we deduce
that the morphism $M_{\sigma}$ (that is the differential at a point
defined over $\mathbb{Q}$ of an automorphism defined over $\mathbb{Q}$)
is defined over $\mathbb{Q}$, thus: $\theta=1$. \\
Let be $\xi=e^{\frac{2i\pi}{11}}$. The equation of the Klein cubic
is chosen in such a way that it is easy to check that the automorphism
\[
f:\,(z_{1}:z_{2}:z_{3}:z_{4}:z_{5})\mapsto(\xi z_{1}:\xi^{9}z_{2}:\xi^{3}z_{3}:\xi^{4}z_{4}:\xi^{5}z_{5})\]
acts on it. Let be $\tau=\rho(f)$. By Theorem \ref{isomorphisme entre aut S et aut F},
the analytic representation of the homomorphism part of $\tau'$ is:
\[
M_{\tau}:(z_{1},z_{2},z_{3},z_{4},z_{5})\mapsto\xi^{j}(\xi z_{1},\xi^{9}z_{2},\xi^{3}z_{3},\xi^{4}z_{4},\xi^{5}z_{5}),\]
where $j\in\mathbb{Z}/11\mathbb{Z}$. Proposition 13.2.5. of  \cite{Birkenhake}
implies that $j=0$.
\begin{notation}
\label{not:For}For $k\in\mathbb{Z}/11\mathbb{Z}$, let $v_{k}$ be
the vector: \begin{eqnarray*}
v_{k} & = & \xi^{k}e_{1}+\xi^{9k}e_{2}+\xi^{3k}e_{3}+\xi^{4k}e_{4}+\xi^{5k}e_{5}\\
 & = & (M_{\tau})^{k}v_{0}\in H^{0}(\Omega_{S})^{*}\end{eqnarray*}
 and let be $\ell_{k}=\xi^{k}x_{1}+\xi^{9k}x_{2}+\xi^{3k}x_{3}+\xi^{4k}x_{4}+\xi^{5k}x_{5}\in H^{0}(\Omega_{S})$.
\end{notation}
Let us construct a sub-lattice of $H_{1}(S,\mathbb{Z})$.\\
Let $q_{1}$ be the homomorphism part of \[
\sum_{k=0}^{k=4}(\sigma')^{k}\]
(where $\sigma'\circ\vartheta=\vartheta\circ\sigma$). Its analytic
representation is:\[
\begin{array}{ccc}
dq_{1}:H^{0}(\Omega_{S})^{*} & \rightarrow & H^{0}(\Omega_{S})^{*}\\
z & \mapsto & \ell_{0}(z)v_{0}.\end{array}\]
 and its image is an elliptic curve which we denote by $\mathbb{E}$.
The restriction of the homomorphism part of $q_{1}\circ\tau':\Alb(S)\rightarrow\mathbb{E}$
to $\mathbb{E}$ is the multiplication by: \[
\nu=\xi+\xi^{9}+\xi^{3}+\xi^{4}+\xi^{5}=\frac{-1+i\sqrt{11}}{2},\]
thus the curve $\mathbb{E}$ has complex multiplication by the principal
ideal domain $\mathbb{Z}[\nu]$ and there is a constant $c\in\mathbb{C}^{*}$
such that :\[
H_{1}(S,\mathbb{Z})\cap\mathbb{C}v_{0}=\mathbb{Z}[\nu]cv_{0}.\]

\begin{rem}
\label{rem:Up-to-a}Up to a normalization of the basis $e_{1},\dots,e_{5}$
by a multiplication by the scalar $\frac{1}{c}$, we suppose that
$c=1$. Under this normalization, the Klein cubic remains the same
: \[
F=\{x_{1}x_{5}^{2}+x_{5}x_{3}^{2}+x_{3}x_{4}^{2}+x_{4}x_{2}^{2}+x_{2}x_{1}^{2}=0\}.\]

\end{rem}
Let $\Lambda_{0}\subset H^{0}(\Omega_{S})^{*}$ be the $\mathbb{Z}$-module
generated by the $v_{k},\, k\in\mathbb{Z}/11\mathbb{Z}$. The group
$\Lambda_{0}$ is stable under the action of $M_{\tau}$ and $\Lambda_{0}\subset H_{1}(S,\mathbb{Z})$.
\begin{lem}
\label{A est isomorphe =0000E0 E5}The $\mathbb{Z}$-module $\Lambda_{0}\subset H_{1}(S,\mathbb{Z})$
is equal to the lattice: \[
R_{0}=\mathbb{Z}[\nu]v_{0}+\mathbb{Z}[\nu]v_{1}+\mathbb{Z}[\nu]v_{2}+\mathbb{Z}[\nu]v_{3}+\mathbb{Z}[\nu]v_{4}.\]
\end{lem}
\begin{proof}
We have:\[
\nu v_{0}=v_{1}+v_{3}+v_{4}+v_{5}+v_{9},\]
hence $\nu v_{0}$ is an element of $\Lambda_{0}$. This implies that
the vectors $\nu v_{k}=(M_{\tau})^{k}\nu v_{0}$ are elements of $\Lambda_{0}$
for all $k$, hence: $R_{0}\subset\Lambda_{0}$. Conversely, we have:
\[
v_{5}=v_{0}+(1+\nu)v_{1}-v_{2}+v_{3}+\nu v_{4}.\]
This proves that the lattice $R_{0}$ contains the vectors $v_{k}=(M_{\tau})^{k}v_{0}$
generating $\Lambda_{0}$, thus: $R_{0}=\Lambda_{0}$. 
\end{proof}
Let us compute the first Chern class $c_{1}(\Theta)$ of the Theta
divisor $\Theta$ of $\Alb(S)$:
\begin{lem}
\label{forme hermitienne de la cub de Klein}Let $H$ be the matrix
of the Hermitian form of the polarization $\Theta$ in the basis $e_{1},\dots,e_{5}$.
There exists a positive integer $a$ such that: \[
H=a\frac{2}{\sqrt{11}}I_{5}\]
 where $I_{5}$ is the size $5$ identity matrix.\end{lem}
\begin{proof}
By Theorem \ref{isomorphisme entre aut S et aut F}, the homomorphism
part of $\tau'$ preserves the polarization $\Theta$. This implies
that: \[
^{t}M_{\tau}H\bar{M}_{\tau}=H\]
where $\bar{M}_{\tau}$ is the matrix in the basis $e_{1},..,e_{5}$
whose coefficients are conjugated of those of $M_{\tau}$. The only
Hermitian matrices that verify this equality are the diagonal matrices.
By the same reasoning with $\sigma$ instead of $\tau$, we obtain
that these diagonal coefficients are equal and: \[
H=a\frac{2}{\sqrt{11}}I_{5}\]
where $a$ is a positive real ($H$ is a positive definite Hermitian
form). As $H$ is a polarization, the alternating form $c_{1}(\Theta)=\Im m(H)$
take integer values on $H_{1}(S,\mathbb{Z})$, hence $a=\Im m(^{t}v_{2}H\bar{v}_{1})$
is an integer.
\end{proof}
Now, we construct a lattice that contains $H_{1}(S,\mathbb{Z})$:\\
Let be $k\in\mathbb{Z}/11\mathbb{Z}$. The analytic representation
of the morphism $q_{1}\circ((\tau')^{k})$ is:\[
\begin{array}{ccc}
H^{0}(\Omega_{S})^{*} & \rightarrow & H^{0}(\Omega_{S})^{*}\\
z & \mapsto & \ell_{k}(z)v_{0}.\end{array}\]
Let be $\lambda\in H_{1}(S,\mathbb{Z})$. As \[
H_{1}(S,\mathbb{Z})\cap\mathbb{C}v_{0}=\mathbb{Z}[\nu]v_{0},\]
the scalar $\ell_{k}(\lambda)$ is an element of $\mathbb{Z}[\nu]$.
Let us define: \[
\Lambda_{4}=\{z\in H^{0}(\Omega_{S})^{*}/\ell_{k}(z)\in\mathbb{Z}[\nu],\,0\leq k\leq4\}.\]

\begin{lem}
\label{le r=0000E9seau lambda 4}The $\mathbb{Z}$-module $\Lambda_{4}\supset H_{1}(S,\mathbb{Z})$
is equal to the lattice: \[
R_{1}=\sum_{k=0}^{k=3}\frac{\mathbb{Z}[\nu]}{1+2\nu}(v_{k}-v_{k+1})+\mathbb{Z}[\nu]v_{0}.\]
Moreover $M_{\tau}$ stabilizes $\Lambda_{4}$.\end{lem}
\begin{proof}
Let be $\ell_{1}^{*},\ldots,\ell_{5}^{*}\in H^{0}(\Omega_{S})^{*}$
be the dual basis of $\ell_{1},\ldots,\ell_{5}$ (see Notations \ref{not:For}).
Then $\Lambda_{4}=\bigoplus_{j=1}^{5}\mathbb{Z}[\nu]\ell_{j}^{*}$.
Since $(e_{1},\ldots,e_{5})=(\ell_{1}^{*},\ldots,\ell_{5}^{*})A$
and $(v_{0},\ldots,v_{4})=(e_{1},\ldots,e_{5})^{t}A$ for the matrix:
\[
A=\left(\begin{array}{ccccc}
1 & 1 & 1 & 1 & 1\\
\xi & \xi^{9} & \xi^{3} & \xi^{4} & \xi^{5}\\
\xi^{2} & \xi^{7} & \xi^{6} & \xi^{8} & \xi^{10}\\
\xi^{3} & \xi^{5} & \xi^{9} & \xi & \xi^{4}\\
\xi^{4} & \xi^{3} & \xi & \xi^{5} & \xi^{9}\end{array}\right),\]
we have $(\ell_{1}^{*},\ldots,\ell_{5}^{*})=(v_{0},\ldots,v_{4})^{t}A^{-1}A$.
Moreover: \[
^{t}A^{-1}A=\frac{1}{1+2\nu}\left(\begin{array}{ccccc}
-1 & -\nu & 0 & -1 & 1-\nu\\
-\nu & 2 & 0 & -\nu & 3+\nu\\
0 & 0 & 0 & 1 & -1\\
-1 & -\nu & 1 & -2 & 1-\nu\\
1-\nu & 3+\nu & -1 & 1-\nu & 2+2\nu\end{array}\right).\]
Let $B$ be the matrix: \[
B=\left(\begin{array}{ccccc}
-\nu-1 & 1 & -1 & 0 & 5\\
1 & -1 & 0 & 0 & \nu\\
-1 & 0 & 0 & 1 & -1-\nu\\
0 & 0 & 1 & 0 & \nu\\
0 & 1 & 0 & 0 & \nu\end{array}\right)\in SL_{5}(\mathbb{Z}[\nu]).\]
We have $(\ell_{1}^{*},\ldots,\ell_{5}^{*})B=(v_{0},\ldots,v_{4})^{t}A^{-1}AB=(\frac{v_{0}-v_{1}}{1+2\nu},\frac{v_{1}-v_{2}}{1+2\nu},\frac{v_{2}-v_{3}}{1+2\nu},\frac{v_{3}-v_{4}}{1+2\nu},v_{0})$,
thus $\Lambda_{4}=R_{1}$. By using the equality: \[
v_{5}=v_{0}+(1+\nu)v_{1}-v_{2}+v_{3}+\nu v_{4},\]
we easily check that the vector $M_{\tau}(\frac{1}{1+2\nu}(v_{3}-v_{4}))\in H^{0}(\Omega_{S})^{*}$
is in $\Lambda_{4}$, hence $\Lambda_{4}$ is stable by $M_{\tau}$.
\end{proof}
Now, using the action of $M_{\tau}$, we determine the lattice $H_{1}(S,\mathbb{Z})$
among lattices $\Lambda$ such that $\Lambda_{0}\subset\Lambda\subset\Lambda_{4}$.\\
 We denote by $\phi:\Lambda_{4}\rightarrow\Lambda_{4}/\Lambda_{0}$
the quotient map. The ring $\mathbb{Z}[\nu]/(1+2\nu)$ is the finite
field with $11$ elements. The quotient $\Lambda_{4}/\Lambda_{0}$
is a $\mathbb{Z}[\nu]/(1+2\nu)$-vector space with basis:\[
\begin{array}{cc}
t_{1}=\frac{1}{1+2\nu}(v_{0}-v_{1})+\Lambda_{0}, & t_{2}=\frac{1}{1+2\nu}(v_{1}-v_{2})+\Lambda_{0},\\
t_{3}=\frac{1}{1+2\nu}(v_{2}-v_{3})+\Lambda_{0}, & t_{4}=\frac{1}{1+2\nu}(v_{3}-v_{4})+\Lambda_{0}.\end{array}\]

Let $R$ be a lattice such that : $\Lambda_{0}\subset R\subset\Lambda_{4}$.
The group $\phi(R)$ is a sub-vector space of $\Lambda_{4}/\Lambda_{0}$
and: \[
\phi^{-1}\phi(R)=R+\Lambda_{0}=R.\]
The set of lattices $R$ such that $\Lambda_{0}\subset R\subset\Lambda_{4}$
is thus in bijection with the set of sub-vector spaces of $\Lambda_{4}/\Lambda_{0}$
and these lattices are also $\mathbb{Z}[\nu]$-modules. \\
Because $M_{\tau}$ preserves $\Lambda_{0}$, the morphism $M_{\tau}$
induces a morphism $\widehat{M}_{\tau}$ on the quotient $\Lambda_{4}/\Lambda_{0}$
such that $\phi\circ M_{\tau}=\widehat{M}_{\tau}\circ\phi$. As $M_{\tau}$
stabilizes $H_{1}(S,\mathbb{Z})$, the vector space $\phi(H_{1}(S,\mathbb{Z}))$
is stable by $\widehat{M}_{\tau}$. Let be:\begin{eqnarray*}
w_{1} & = & -t_{1}+3t_{2}-3t_{3}+t_{4}\\
w_{2} & = & t_{1}-2t_{2}+t_{3}\\
w_{3} & = & -t_{1}+t_{2}\\
w_{4} & = & t_{1}.\end{eqnarray*}
 The matrix of $\widehat{M}_{\tau}$ is \[
\left(\begin{array}{cccc}
0 & 0 & 0 & -1\\
1 & 0 & 0 & 4\\
0 & 1 & 0 & 5\\
0 & 0 & 1 & 4\end{array}\right)\]
 in the basis $t_{1},\ldots,t_{4}$ of $\Lambda_{4}/\Lambda$ and
the matrix of $\widehat{M}_{\tau}$ is:\[
\left(\begin{array}{cccc}
1 & 1 & 0 & 0\\
0 & 1 & 1 & 0\\
0 & 0 & 1 & 1\\
0 & 0 & 0 & 1\end{array}\right)\]
in the basis $w_{1},.\ldots,w_{4}$. The sub-spaces stable by $\widehat{M}_{\tau}$
are the space $W_{0}=\{0\}$ and the spaces $W_{j},\,1\leq j\leq4$
generated by $w_{1},..,w_{j}$. Let $\Lambda_{j}$ be the lattice
$\phi^{-1}W_{j}$, then:
\begin{thm}
\label{thmThe-attice}The lattice $H_{1}(S,\mathbb{Z})$ is equal
to $\Lambda_{2}$, and $\Lambda_{2}$ is equal to \[
R_{2}=\frac{\mathbb{Z}[\nu]}{1+2\nu}(v_{0}-3v_{1}+3v_{2}-v_{3})+\frac{\mathbb{Z}[\nu]}{1+2\nu}(v_{1}-3v_{2}+3v_{3}-v_{4})+\bigoplus_{k=0}^{2}\mathbb{Z}[\nu]v_{k}.\]
Moreover, the Hermitian matrix associated to $\Theta$ is equal to
$\frac{2}{\sqrt{11}}I_{5}$ in the basis $e_{1},\dots,e_{5}$ and
$c_{1}(\Theta)=\frac{i}{\sqrt{11}}\sum_{k=1}^{5}dx_{k}\wedge d\bar{x}_{k}$.\end{thm}
\begin{proof}
Let $c_{1}(\Theta)=\Im m(H)=i\frac{a}{\sqrt{11}}\sum dx_{k}\wedge d\bar{x}_{k}$
be the alternating form of the principal polarization $\Theta$. Let
$\lambda_{1},\dots,\lambda_{10}$ be a basis of a lattice $\Lambda$.
By definition, the square of the Pfaffian $Pf_{\Theta}(\Lambda)$
of $\Lambda$ is the determinant of the matrix \[
\left(c_{1}(\Theta)(\lambda_{j},\lambda_{k})\right)_{1\leq j,k\leq10}.\]
 Since $\Theta$ is a principal polarization, we have $Pf_{\Theta}(H_{1}(S,\mathbb{Z}))=1$.\\
It is easy to find a basis of $\Lambda_{j}$ ($j\in\{0,\dots,4\}$).
For example, the space $W_{2}$ is generated by $w_{2}=t_{1}-2t_{2}+t_{3}$
and $w_{1}+w_{2}=t_{2}-2t_{3}+t_{4}$ and as \[
\begin{array}{c}
\phi(\frac{1}{1+2\nu}(v_{0}-3v_{1}+3v_{2}-v_{3}))=w_{2},\\
\phi(\frac{1}{1+2\nu}(v_{1}-3v_{2}+3v_{3}-v_{4}))=w_{1}+w_{2},\end{array}\]
 the lattice $R_{2}$ (that contains $\Lambda_{0}$) is equal to $\Lambda_{2}$.
\\
Then, with the help of a computer, we can calculate the square
of the Pfaffian $P_{j}$ of the lattice $\Lambda_{j}$ and verify
that it is equal to:\[
a^{10}11^{4-2j}\]
where $a$ is the integer of Lemma \ref{forme hermitienne de la cub de Klein}.
As $a$ is a positive integer, the only possibility that $P_{j}$
equals $1$ is $j=2$ and $a=1$.\end{proof}
\begin{rem}
\label{quartique de Klein}Let $C$ be the Klein quartic curve : this
curve is canonically embedded into $\mathbb{P}^{2}=\mathbb{P}(H^{0}(C,\Omega_{C})^{*})$
and there exists a basis $x_{1},x_{2},x_{3}$ of $H^{0}(C,\Omega_{C})$
such that $C=\{x_{1}^{3}x_{2}+x_{2}^{3}x_{3}+x_{3}^{3}x_{1}=0\}$.
The automorphism group of $C$ is $PSL_{2}(\mathbb{F}_{7})$. By taking
exactly the same arguments as the Klein cubic threefold, it is possible
to compute the period lattice $H_{1}(C,\mathbb{Z})\subset H^{0}(C,\Omega_{C})^{*}$. 
\end{rem}

\section{The Néron-Severi group of the Fano surface of the Klein cubic.}

Let us define: \[
\begin{array}{c}
u_{1}=\frac{1}{1+2\nu}(v_{0}-3v_{1}+3v_{2}-v_{3}),\, u_{2}=\frac{1}{1+2\nu}(v_{1}-3v_{2}+3v_{3}-v_{4}),\\
u_{3}=v_{0},\; u_{4}=v_{1},\; u_{5}=v_{2},\end{array}\]
and let $y_{1},\dots,y_{5}\in H^{0}(\Omega_{S})$ be the dual basis
of $u_{1},\dots,u_{5}.$ \\
Let be $k,\,1\leq k\leq5$. The image of $H_{1}(S,\mathbb{Z})$
by $y_{k}\in H^{0}(\Omega_{S})$ is $\mathbb{Z}[\nu]$, and this form
is the analytic representation of a morphism of Abelian varieties
\[
r_{k}:\Alb(S)\rightarrow\mathbb{E}=\mathbb{C}/\mathbb{Z}[\nu].\]
By Theorem \ref{thmThe-attice}, the morphisms $r_{1},\dots,r_{5}$
form a basis of the $\mathbb{Z}[\nu]$-module of rank $5$ of homomorphisms
between $\Alb(S)$ and $\mathbb{E}$.

We denote by $\Lambda_{A}^{*}$ the free $\mathbb{Z}[\nu]$-module
of rank $5$ generated by $y_{1},\dots,y_{5}$ and for $\ell\in\Lambda_{A}^{*}\setminus\{0\}$,
we denote by $\Gamma_{\ell}:\Alb(S)\rightarrow\mathbb{E}$ the morphism
whose analytic representation is $\ell:H^{0}(\Omega_{S})^{*}\rightarrow\mathbb{C}$.\\
Let $\vartheta:S\rightarrow\Alb(S)$ be a fixed Albanese morphism.
We denote by  $\gamma_{\ell}:S\rightarrow\mathbb{E}$ the morphism
$\gamma_{\ell}=\Gamma_{\ell}\circ\vartheta$ and we denote by  $F_{\ell}$
the numerical equivalence class of a fibre of $\gamma_{\ell}$ (this
class is independent of the choice of $\vartheta$).\\
We define the scalar product of two forms $\ell,\ell'\in\Lambda_{A}^{*}$
by: \[
\left\langle \ell,\ell'\right\rangle =\sum_{k=1}^{k=5}\ell(e_{k})\overline{\ell'(e_{k})}\]
 and the norm of $\ell$ by:\[
\left\Vert \ell\right\Vert =\sqrt{\left\langle \ell,\ell\right\rangle }.\]
We denote by $\NS(X)$ the Néron-Severi group of a variety $X$. For
a point $s$ of $S$, we denote by $C_{s}$ the incidence divisor
that parametrizes the lines on $F$ that cut the line corresponding
to the point $s$. The aim of this section is to prove the following
result: 
\begin{thm}
\label{la forme Q pour 11}1) Let $\ell,\ell'$ be non-zero elements
of $\Lambda_{A}^{*}$. The fibre $F_{\ell}$ has arithmetic genus:\[
g(F_{\ell})=1+3\left\Vert \ell\right\Vert ^{2},\]
satisfies $C_{s}F_{\ell}=2\left\Vert \ell\right\Vert ^{2}$ and :
\[
F_{\ell}F_{\ell'}=\left\Vert \ell\right\Vert ^{2}\left\Vert \ell'\right\Vert ^{2}-\left\langle \ell,\ell'\right\rangle \left\langle \ell',\ell\right\rangle .\]
2) The image of the morphism $\vartheta^{*}:\NS(\Alb(S))\rightarrow\NS(S)$
is a rank $25$ sub-lattice of discriminant $2^{2}11^{10}$. \\
3) The following $25$ fibres \[
\left\{ \begin{array}{cc}
F_{y_{k}} & k\in\{1,\dots,5\},\\
F_{y_{k}+y_{l}} & 1\leq k<l\leq5,\\
F_{y_{k}+\nu y_{l}} & 1\leq k<l\leq5,\end{array}\right.\]
form a $\mathbb{Z}$-basis of $\vartheta^{*}\NS(\Alb(S))$ and together
with the class of the incident divisor $C_{s}$ ($s\in S$), they
generate $\NS(S)$. The lattice $\NS(S)$ has discriminant $11^{10}$. 
\end{thm}
We identify elements of the Néron-Severi group of $\Alb(S)$ with
alternating forms.
\begin{lem}
\label{le NS de A de klein}The Néron-Severi group of $\Alb(S)$ is
generated by the $25$ forms:

\[
\left\{ \begin{array}{cc}
\frac{i}{\sqrt{11}}dy_{k}\wedge d\bar{y}_{k} & k\in\{1,\dots,5\},\\
\frac{i}{\sqrt{11}}(dy_{k}\wedge d\bar{y}_{l}+dy_{l}\wedge d\bar{y}_{k}) & 1\leq k<l\leq5,\\
\frac{i}{\sqrt{11}}(\nu dy_{k}\wedge d\bar{y}_{l}+\overline{\nu}dy_{l}\wedge d\bar{y}_{k}), & 1\leq k<l\leq5.\end{array}\right.\]
\end{lem}
\begin{proof}
The Hermitian form $H'=\frac{2}{\sqrt{11}}I_{5}$ in the basis $u_{1},\ldots,u_{5}$
defines a principal polarization of $\Alb(S)$. Let $\End^{s}(\Alb(S))$
be the group of symmetrical morphisms for the Rosati involution associated
to $H'$. An endomorphism of $\Alb(S)$ can be represented by a matrix
$A\in M_{5}(\mathbb{Z}[\nu])$ in the basis $u_{1},\ldots,u_{5}$.
The symmetrical endomorphisms satisfy $^{t}AH'=H'\bar{A}$ i.e. $^{t}A=\bar{A}$.
A basis $\mathcal{B}$ of the group of symmetrical elements is :\[
\left\{ \begin{array}{cc}
e_{kk} & k\in\{1,\dots,5\},\\
e_{kl}+e_{lk} & 1\leq k<l\leq5,\\
\nu e_{kl}+\overline{\nu}e_{lk} & 1\leq k<l\leq5,\end{array}\right.\]
where $e_{kl}$ is the matrix with entry $1$ in the intersection
of line $k$ and row $l$ and $0$ elsewhere.\\
By \cite{Birkenhake}, Proposition $5.2.1$ and Remark 5.2.2.,
the map:\[
\begin{array}{ccc}
\phi_{H'}:\End^{s}(\Alb(S)) & \rightarrow & \NS(\Alb(S))\\
A & \mapsto & \Im m(\cdot^{t}AH'\bar{\cdot})\end{array}\]
is an isomorphism of groups. We obtain the base of the Lemma by taking
the image by $\phi_{H'}$ of the base $\mathcal{B}$.
\end{proof}
The Néron-Severi group of the curve $\mathbb{E}=\mathbb{C}/\mathbb{Z}[\nu]$
is the $\mathbb{Z}$-module generated by the form $\eta=\frac{i}{\sqrt{11}}dz\wedge d\bar{z}$.
Let be $\ell\in\Lambda_{A}^{*}\setminus\{0\}$. We have: \[
\Gamma_{\ell}^{*}\eta=\frac{i}{\sqrt{11}}d\ell\wedge d\overline{\ell}\]
and this form is the Chern class of the divisor $\Gamma_{\ell}^{*}0$.
\begin{lem}
\label{base + simple}The $25$ forms: \[
\left\{ \begin{array}{cc}
\eta_{k}=\Gamma_{y_{k}}^{*}\eta & k\in\{1,\dots,5\},\\
\eta_{k,l}^{1}=\Gamma_{y_{k}+y_{l}}^{*}\eta & 1\leq k<l\leq5,\\
\eta_{k,l}^{\nu}=\Gamma_{y_{k}+\nu y_{l}}^{*}\eta & 1\leq k<l\leq5.\end{array}\right.\]
 are a basis of the Néron-Severi group of $\Alb(S)$.\end{lem}
\begin{proof}
Let $1\leq k\leq5$ be an integer. The element $\Gamma_{y_{k}}^{*}\eta=\frac{i}{\sqrt{11}}dy_{k}\wedge d\bar{y}_{k}$
lies in the basis of Lemma \ref{le NS de A de klein}. Let $1\leq l<k\leq5$
be integers, let be $a\in\{1,\nu\}$, and $\ell=y_{k}+ay_{l}$. We
have: \[
\Gamma_{\ell}^{*}\eta=\frac{i}{\sqrt{11}}(dy_{k}\wedge d\bar{y}_{k}+\bar{a}dy_{k}\wedge d\bar{y}_{l}+ady_{l}\wedge d\bar{y}_{k}+a\bar{a}dy_{l}\wedge d\bar{y}_{l}),\]
this proves, when we take $a=1$ and next $a=\nu$, that the forms
of the basis of Lemma \ref{le NS de A de klein} are $\mathbb{Z}$-linear
combinations of the forms $\eta_{k},\eta_{k,l}^{1},\eta_{k,l}^{\nu},\,1\leq k,l\leq5$.
\end{proof}
Let us prove Theorem \ref{la forme Q pour 11}.
\begin{proof}
By \cite{Clemens}, the homology class of $S$ in $\Alb(S)$ is equal
to $\frac{\Theta^{3}}{3!}$, thus the intersection of the fibres $F_{\ell}$
and $F_{\ell'}$ is equal to: \[
\int_{A}\frac{1}{3!}\wedge^{3}c_{1}(\Theta)\wedge\Gamma_{\ell}^{*}\eta\wedge\Gamma_{\ell'}^{*}\eta.\]
Write $\ell$ in the basis $x_{1},..,x_{5}$ : $\ell=a_{1}x_{1}+\dots+a_{5}x_{5}$
and $\ell'=b_{1}x_{1}+\dots+b_{5}x_{5}$, then: \[
\frac{1}{3!}(\frac{i}{\sqrt{11}})^{2}d\ell\wedge d\overline{\ell}\wedge d\ell'\wedge d\overline{\ell'}\wedge(\wedge^{3}c_{1}(\Theta))\]
 is equal to: \[
\begin{array}{c}
(\frac{i}{\sqrt{11}})^{5}(\sum a_{j}dx_{j})\wedge(\sum\bar{a}_{j}d\bar{x}_{j})\wedge(\sum b_{j}dx_{j})\wedge(\sum\bar{b}_{j}d\bar{x}_{j})\\
\wedge\sum_{h<j<k}dx_{h}\wedge d\bar{x}_{h}\wedge dx_{j}\wedge d\bar{x}_{j}\wedge dx_{k}\wedge d\bar{x}_{k}\end{array}\]
that is equal to: \[
(\sum_{k\not=j}(a_{k}\bar{a}_{k}b_{j}\bar{b}_{j}-a_{k}\bar{a}_{j}b_{j}\bar{b}_{k}))\frac{1}{5!}\wedge^{5}c_{1}(\Theta).\]
 But : $\int_{A}\frac{1}{5!}\wedge^{5}c_{1}(\Theta)=1$ because $\Theta$
is a principal polarization of $\Alb(S)$, hence:\[
\begin{array}{ccc}
F_{\ell}F_{\ell'}=\int_{A}\frac{1}{3!}\wedge^{3}c_{1}(\Theta)\wedge\Gamma_{\ell}^{*}\eta\wedge\Gamma_{\ell'}^{*}\eta & =\sum_{k\not=j}(a_{k}\bar{a}_{k}b_{j}\bar{b}_{j}-a_{k}\bar{a}_{j}b_{j}\bar{b}_{k})\\
 & =\left\Vert \ell\right\Vert ^{2}\left\Vert \ell'\right\Vert ^{2}-\left\langle \ell,\ell'\right\rangle \left\langle \ell',\ell\right\rangle .\end{array}\]
 By \cite{Clemens} (10.9) and Lemma 11.27, $\frac{3}{2}\vartheta^{*}c_{1}(\Theta)$
is the Poincaré dual of a canonical divisor $K$ of $S$, hence:\[
KF_{\ell}=\frac{3}{2}\vartheta^{*}c_{1}(\Theta)\vartheta^{*}\Gamma_{\ell}^{*}\eta=\frac{3}{2}\int_{A}\frac{1}{3!}\wedge^{4}c_{1}(\Theta)\wedge\Gamma_{\ell}^{*}\eta\]
and: \[
KF_{\ell}=\int_{A}6(\frac{i}{\sqrt{11}})^{5}(\sum a_{j}dx_{j})\wedge(\sum\bar{a}_{j}d\bar{x}_{j})\wedge\sum_{1\leq k\leq5}(\bigwedge_{j\not=k}(dx_{j}\wedge d\bar{x}_{j}))\]
so $KF_{\ell}=6\sum_{k=1}^{k=5}a_{k}\bar{a}_{k}=6\left\Vert \ell\right\Vert ^{2}$.
Thus we have: $g(F_{\ell})=(KF_{\ell}+0)/2+1=3\left\Vert \ell\right\Vert ^{2}+1$.
By \cite{Clemens}, $3C_{s}$ is numerically equivalent to $K$, hence
$C_{s}F_{\ell}=2\left\Vert \ell\right\Vert ^{2}$.

Lemma \ref{base + simple} gives us a basis $\eta_{1},...,\eta_{25}$
of $\NS(\Alb(S))$ and we know the intersections $\vartheta^{*}\eta_{k}\vartheta^{*}\eta_{l}$
in the Fano surface. With the help of a computer, we can verify that
the determinant of the intersection matrix: \[
(\vartheta^{*}\eta_{k}\vartheta^{*}\eta_{l})_{1\leq k,l\leq25}\]
is equal to $2^{2}11^{10}$. By general results of \cite{Roulleau},
the index of $\vartheta^{*}\NS(\Alb(S))\subset\NS(S)$ is $2$ and
$\NS(S)$ is generated by $\vartheta^{*}\NS(\Alb(S))$ and the class
of an incidence divisor $C_{s}$.\end{proof}
\begin{cor}
\label{corLetC}1) Let $C$ be a smooth curve of genus $>0$ and let
$\gamma:S\rightarrow C$ be a fibration with connected fibres. Then
there exists an isomorphism $j:\mathbb{E}\rightarrow C$ and a form
$\ell\in\Lambda_{A}^{*}$ such that $\gamma=j\circ\gamma_{\ell}$.
\\
2) The set of numerical classes of fibres of connected fibrations
of $S$ onto a curve of positive genus is in natural bijection with
$\mathbb{P}^{4}(\mathbb{Q}(\nu))$.\\
3) The fibres of these fibrations generate $\vartheta^{*}\NS(\Alb(S))$.
\end{cor}
To prove Corollary \ref{corLetC}, we need the following Lemma:
\begin{lem}
\label{lem:Let-be-}1) Let be $\ell\in\Lambda_{A}^{*}\setminus\{0\},\,\ell=t_{1}y_{1}+\dots+t_{5}y_{5}$.
The fibration $\Gamma_{\ell}$ has connected fibres if and only if
$t_{1},\dots,t_{5}$ generate $\mathbb{Z}[\nu]$.\\
2) Let $\Gamma:\Alb(S)\rightarrow C$ be a morphism with connected
fibres onto an elliptic curve $C$. Then $C=\mathbb{E}$ and there
exists $\ell\in\Lambda_{A}^{*}$ such that $\Gamma=\Gamma_{\ell}$.\end{lem}
\begin{proof}
Let ${\bf t}\subset\mathbb{Z}[\nu]$ be the ideal of $\mathbb{Z}[\nu]$
generated by $t_{1},\ldots,t_{5}$. This ideal satisfies \[
d\Gamma_{\ell}(H_{1}(S,\mathbb{Z}))={\bf t}.\]
The morphism $\Gamma_{\ell}$ factorizes through the natural morphisms
: $\Alb(S)\rightarrow\mathbb{C}/{\bf t}$ and $\mathbb{C}/{\bf t}\rightarrow\mathbb{E}=\mathbb{C}/\mathbb{Z}[\nu]$.
\\
If ${\bf t}\not=\mathbb{Z}[\nu]$, then the fibres of $\Gamma_{\ell}$
are not connected because the fibres of $\mathbb{C}/{\bf t}\rightarrow\mathbb{E}$
are not connected.\\
Let us recall that $\mathbb{Z}[\nu]$ is a principal ideal domain
: there exist a generator $g$ of ${\bf t}$ and $\ell'\in\Lambda_{A}^{*}$
such that $\ell=g\ell'$. If we replace $\ell$ by $\ell'$, we are
now reduced to the case where ${\bf t}=\mathbb{Z}[\nu]$, $g=1$.
\\
In that case, there exist $a_{1},\ldots,a_{5}\in\mathbb{Z}[\nu]$
such that $\sum a_{i}t_{i}=1$. The homomorphism $\mathbb{E}\rightarrow\Alb(S)$
whose analytic representation is:\[
\begin{array}{ccc}
\mathbb{C} & \rightarrow & H^{0}(\Omega_{S})^{*}\\
z & \mapsto & \sum a_{i}u_{i}\end{array}\]
is a section of $\Gamma_{\ell}$, hence : $\Alb(S)\simeq\mathbb{E}\times\Ker(\Gamma_{\ell})$,
and since $\Alb(S)$ is connected, that implies that the fibre $\Ker(\Gamma_{\ell})$
is connected ; thus $\Gamma_{\ell}$ has connected fibres.\\
Now let be $\Gamma$ as in part 2). The curve $C$ is isogenous
to $\mathbb{E}$ ; let $j:C\rightarrow\mathbb{E}$ be an isogeny.
There exists $\ell$ such that $j\circ\Gamma=\Gamma_{\ell}$. Moreover,
we proved that there exists $\ell'$ such that $\Gamma_{\ell'}$ is
the Stein factorization of $\Gamma_{\ell}=j\circ\Gamma$. As $\Gamma$
and $\Gamma_{\ell'}$ have the same fibres, we see that \[
C=\Alb(S)/\Ker(\Gamma)=\Alb(S)/\Ker(\Gamma_{\ell'})=\mathbb{E}\]
 and $\Gamma=\Gamma_{\ell'}$.
\end{proof}
Let us prove Corollary \ref{corLetC}:
\begin{proof}
Let $\gamma:S\rightarrow C$ be a fibration onto a curve of genus
$>0$. Since the natural morphism $\wedge^{2}H^{0}(\Omega_{S})\rightarrow H^{0}(S,\wedge^{2}\Omega_{S})$
is an isomorphism \cite{Clemens}, the Castelnuovo Lemma implies that
the curve $C$ has genus $1$. \\
Let $\Gamma:\Alb(S)\rightarrow C$ be the morphism such that $\gamma=\Gamma\circ\vartheta$.
The fibres of $\Gamma$ are connected, otherwise their trace on $S\hookrightarrow\Alb(S)$
would be disconnected fibres of $\gamma$. \\
Lemma \ref{lem:Let-be-} 2) implies that $C=\mathbb{E}$ and there
exists a $\ell$ such that $\Gamma=\Gamma_{\ell}$. Moreover, this
$\ell$ satisfies $\ell(H_{1}(S,\mathbb{Z}))=\mathbb{Z}[\nu]$ and
thus defines a point in $\mathbb{P}_{\mathbb{Z}}^{4}(\mathbb{Z}[\nu])$
; this last set is canonically identified with $\mathbb{P}^{4}(\mathbb{Q}(\nu))$.
Thus, we proved that to the numerical class of a fibre of a connected
fibration $\gamma$, we can associate a point of $\mathbb{P}^{4}(\mathbb{Q}(\nu))$
(the numerical class of a fibre determine the fibration).\\
Conversely, by Lemma \ref{lem:Let-be-} 1), to a point of $\mathbb{P}^{4}(\mathbb{Q}(\nu))$,
there corresponds a form $\ell\in\Lambda_{A}^{*}$ (up to a sign)
such that $\Gamma_{\ell}$ has connected fibres. Let $\gamma$ be
the Stein factorization of $\Gamma_{\ell}\circ\vartheta$. We proved
that there is a form $\ell'$ such that $\gamma=\Gamma_{\ell'}\circ\vartheta$
and $\Gamma_{\ell'}$ has connected fibres. Thus $\ell=\ell'$, $\Gamma_{\ell}\circ\vartheta$
has connected fibres and to $\ell$ we associate the fibre $F_{\ell}\in\NS(S)$.
This class $F_{\ell}$ is independent of the choice of $\pm\ell$.
That ends the proof of parts 1) and 2).\\
The point 3) is a reformulation of Lemma \ref{base + simple}.
\end{proof}

Xavier Roulleau, \\
Graduate School of Mathematical Sciences,\\
University of Tokyo,\\
3-8-1 Komaba, Meguro,\\
 Tokyo, 153-8914\\
 Japan\\
 roulleau@ms.u-tokyo.ac.jp %

\end{document}